\documentclass[11pt]{amsart}
\usepackage{graphicx,amssymb,amscd,amsthm,verbatim}

\newtheorem{proposition}{Proposition}
\newtheorem{theorem}[proposition]{Theorem}

\newtheorem{lemma}[proposition]{Lemma}

\newcommand{\reals}{\mathbb R}

\def \ignore#1{\relax}

\begin{document}

\title[The Grone Merris Conjecture] {The Grone-Merris Conjecture}

\author{Hua Bai}
\address{Department of Mathematics\\
         Boston College\\
         Chestnut Hill, MA 02467, USA}
\email{baihu@bc.edu}

\date{\today}
\thanks{The author was partially supported by NSF grant DMS-0604866}


\maketitle

\begin{abstract}
{In spectral graph theory, the Grone and Merris conjecture asserts
that the spectrum of the Laplacian matrix of a finite graph is
majorized by the conjugate degree sequence of this graph. We give
a complete proof for this conjecture.}
\end{abstract}

The Laplacian of a simple graph $G $ with $n$ vertices is a
positive semi-definite $n\times n$ matrix $L(G)$ that mimics the
geometric Laplacian of a Riemannian manifold; see \S
\ref{sec:Definitions} for definitions, and \cite{Chung97,
Merris9402} for comprehensive bibliographies on the graph
Laplacian. The spectrum sequence $\lambda(G)$ of $L(G)$ can be
listed in non-increasing order as
$$
\lambda_1(G)\ge \lambda_2(G)\ge \dots \ge \lambda_{n-1}(G)\ge
\lambda_n(G) = 0 .
$$

For two non-increasing real sequences $\mathbf{x}$ and
$\mathbf{y}$ of length $n$, we say that $\mathbf{x}$ is
\emph{majorized} by $\mathbf{y}$ (denoted $\mathbf{x}\preccurlyeq
\mathbf{y}$) if
$$
\sum_{i=1}^k x_i \le \sum_{i=1}^k y_i  \text{ for all } k \le n,
 \text{ and  }  \sum_{i=1}^n x_i=\sum_{i=1}^n y_i.
$$
This notion was introduced because of the following fundamental theorem.

\begin{theorem}[Schur-Horn Dominance Theorem \cite{Schur23,Horn54}]
There exists a Hermitian matrix $H$ with diagonal entry sequence
$\mathbf{x}$ and spectrum sequence $\mathbf{y}$ if and only if $
\mathbf{x} \preccurlyeq \mathbf{y} $. \qed
\end{theorem}

In particular, if $\mathbf{d}(G) = (d_1,d_2,\dots, d_n)^T$ is the
non-increasing \textit{degree sequence} of $G$, which coincides
the diagonal entry sequence of the Laplacian matrix $L(G)$, the
Schur-Horn Dominance Theorem implies that $\mathbf{d}(G)
\preccurlyeq \lambda(G)$. Grone \cite{Grone95} improves this
majorization result: if $G$ has at least one edge, then $ (d_1+1,
d_2, \dots, d_{n-1}, d_n-1)^T \preccurlyeq \lambda(G)$.

For a non-negative integral sequence $\mathbf{d}$, we define its
\textit{conjugate degree sequence} as the sequence $\mathbf{d}' =
(d_1',d_2',\ldots, d_n')^T $ where
$$
d_k' :=\# \{i: d_i \ge k \}.
$$
Another important majorization relation is the following.

\begin{theorem}[Gale-Ryser \cite{Gale57, Ryser57}]
There exists a $(0,1)$-matrix $A$ with row and column sum vectors
$\mathbf{r}$ and $\mathbf{c}$ if and only if $\mathbf{r}
\preccurlyeq\mathbf{c'}$. \qed
\end{theorem}

Applying this to the adjacency matrix of $G$ immediately gives
that  ${\mathbf d}(G) \preccurlyeq{\mathbf d'}(G)$.

In 1994, Grone and Merris \cite{GroneMerris90,GroneMerris94}
raised the natural question whether the Laplacian spectrum
sequence and the conjugate degree sequence are majorization
comparable.

\newtheorem*{GMconjecture}{Grone-Merris Conjecture}

\begin{GMconjecture}
For any graph $G$, the Laplacian spectrum is majorized by the
conjugate degree sequence
$$
 \lambda(G) \preccurlyeq {\mathbf d'}(G).
$$
\end{GMconjecture}

In this paper, we give a complete proof to the Grone-Merris
Conjecture. As a consequence, we have the double majorization
$\mathbf d(G) \preccurlyeq \lambda(G) \preccurlyeq {\mathbf
d'}(G)$.

\medskip

See \cite{DuvalReiner02} for a partial result in this direction,
as well as \cite{Stephen05, Katz05, Kirk09, BLP08} for proofs in
the special cases. See also  \cite{DuvalReiner02} for a
generalization to simplicial complexes, which is still open.

\medskip

\noindent\textbf{Acknowledgements:} This work was started while
the author was visiting the University of Southern California,
whose support and hospitality is gratefully acknowledged. The
author also thanks Francis Bonahon for his support and
encouragement throughout the years, Jun Ying and Jie Ying for
critical Matlab computations, Russell Merris for useful
suggestions, and Andries Brouwer, Tao Li and the referee for many
valuable comments.

\section {The Laplacian matrix and the majorization relation}
\label{sec:Definitions}

Let $G=(V,E)$ be a simple finite graph with $n=|V|$
vertices. We write $i\sim j$ when the  $i$-th vertex is adjacent
to the $j$-th vertex, and we let $d_i$ denote the degree of the $i$-th
vertex.

The \textit{Laplacian matrix} $L(G)$ of the graph $G$ is the
$n\times n$ matrix defined by
$$
   L(G)_{ij}=\left\{
                 \begin{array}{ll}
                     d_i & \mbox{if\ }  i=j; \\
                     -1  & \mbox{if\ }  i\sim j;\\
                     0   & \mbox{otherwise}.
                 \end{array}
             \right.
$$
We can also express the Laplacian as $L(G)= D - A $, where $D$ is
the diagonal matrix defined by the  degree sequence, and $A$ is
the adjacency
 $(0,1)$-matrix of the graph.

It is well-known that $L(G)$ is positive semi-definite, since it
corresponds to the quadratic form
$$
 \mathbf{x}^T L(G) \mathbf{x} = \sum_{i \sim j} (x_i - x_j)^2
  \text{ for  }  \mathbf{x}=(x_1, \dots, x_n)^T \in \reals^n.
$$
Let $ \mathbf{\lambda} = (\lambda_1,\lambda_2, \dots,
\lambda_n)^T$ be the non-increasing \textit{spectrum sequence} of
the Laplacian matrix $L(G)$. The smallest eigenvalue is
$\lambda_n=0$, with eigenvector $\mathbf{1}_n =  (1,1,\dots, 1 )^T
$.

Given two vectors $ \mathbf{x} = (x_1, \dots, x_n)^T $ and $
\mathbf{y} = (y_1, \dots, y_n)^T $ in $\reals^n$, rearrange their
components in non-increasing order as
$$
x_{[1]} \ge x_{[2]} \ge \dots \ge x_{[n]}, \quad y_{[1]} \ge
y_{[2]} \ge \dots \ge y_{[n]}.
$$
We say that $\mathbf{x}$ is \textit{majorized} by $\mathbf{y}$,
and write $\mathbf{x} \preccurlyeq \mathbf{y}$, if
$$
  \sum_{i=1}^k x_{[i]} \le \sum_{i=1}^k y_{[i]}  \text{ for all }
  1 \le k \le n, \text{ and } \sum_{i=1}^n x_i = \sum_{i=1}^n y_i.
$$
We will make use of the following majorization inequality.

\begin{theorem}
[Fan  \cite{Fan49}] \label{thm: Fan Ky}
If $H_1$ and $H_2$ are
Hermitian matrices,  then
$$
\lambda(H_1+H_2) \preccurlyeq \lambda(H_1) + \lambda(H_1).
$$
\vskip -\belowdisplayskip
\vskip -\baselineskip
  \qed
\end{theorem}

\section{Split graphs}

A graph is \textit{split} (also called \textit{semi-bipartite} in
\cite{Katz05}) if its vertices can be partitioned into a clique
$V_1$ and a  co-clique $V_2$. This is equivalent to saying that
the subgraph induced by $V_1$ is complete, and that the subgraph
induced by $V_2$ is an independent set. See \cite{FolderHammer77,
TyshkevichChernyak79, Merris03, HammerSimeone81} for many
characterizations and properties of split graphs.

Given a split graph $G=(V,E)$, let $N=|V_1|$ be the size of
the clique, and $M=|V_2|$ be the size of the co-clique.  Let $\delta(G)$
be the maximum degree of vertices in $V_2$.  Clearly $\delta(G)\le
N$, and the Laplacian matrix of the split graph $G$ is of the form
$$
  L(G)=\left(
          \begin{array} {cc}
             K_N + D_1 & -A  \\
             -A^T      & D_2
          \end{array}
     \right),
$$
where $K_N$ is the Laplacian matrix of the complete graph on $N$
vertices, where $D_1$ and $D_2$ are diagonal matrices with
diagonal entries the vertex degrees  of $V_1$, $V_2$ respectively,
and where $A$ is the adjacency matrix for edges between $V_1$ and
$V_2$.

The Laplacian matrix is symmetric, and therefore Hermitian.

\begin{theorem}[Courant-Fischer-Weyl \cite{ReedSimon78}]
\label{thm: CFW}
Let the $n\times n$ matrix $H$ be  Hermitian,
with eigenvalues $\lambda_1 \ge \lambda_2 \ge \dots \ge
\lambda_n$. Then
$$
\lambda_k= \max_{\mathrm{dim}(S)=k} \  \min_{0 \ne x\in
S}  \frac{ \langle Hx,x \rangle } {\langle x,x
\rangle}
 =\min_{\mathrm{dim}(S)={n-k+1}} \  \max_ {0 \ne x\in S}
   \frac{ \langle Hx,x \rangle} {\langle x,x
\rangle},
$$
where  the $\max$  (resp. $\min$) is taken over all $k$-dimensional
(resp. $(n-k+1)$-dimensional) subspaces of $\reals^n$. \qed
\end{theorem}

We first investigate  the Laplacian spectrum of a split graph.

\begin{proposition} \label{prop: simple fact}
If $G$ is a  split graph of clique size $N$, then
$$
\lambda_{N-1}(G) \ge N \ge \delta(G) \ge \lambda_{N+1}(G).
$$
Moreover, if $ \lambda_{N}(G) \ge N $, then
$$
\sum_{i=1}^N d_i'  =  N^2 +\mathrm{Tr}(D_1).
$$
\end{proposition}

\begin{proof}
To prove the inequalities involving $\lambda_{N-1}(G)$ and
$\lambda_{N+1}(G)$ by the Courant-Fischer-Weyl Min-Max Principle,
it suffices to find an $(N-1)$-dimensional  (resp.
$M$-dimensional) subspace for which the action of $L(G)$ has a
desirable lower  (resp. upper) bound. There are natural
candidates.

Let $P \subset \mathbb{R}^{M+N}$ be the $(N-1)$-dimensional
subspace  consisting of all vectors of the form
$\left(
     \begin{array} {c}
      u \\ \mathbf{0}_{M}
     \end{array}
 \right)
$ with $u\in \mathbb{R}^{N}$ and $u \perp \mathbf{1}_N$. Then for
any unit vector $u\in \reals^n$,
$$
   \left\langle L(G) \left(
                              \begin{array} {c}
                                 u \\  \mathbf{0}_{M}
                              \end{array}
                            \right),
                           \left(
                              \begin{array} {c}
                                 u \\  \mathbf{0}_{M}
                               \end{array}
                            \right)
   \right\rangle  =  \langle (K_N+ D_1) u, u \rangle
                 = N+ \langle  D_1 u, u \rangle
                 \ge N.
$$

Similarly, consider the $M$-dimensional subspace $Q \subset
\mathbb{R}^{M+N}$ consisting of all vectors of the form
$\left(
     \begin{array} {c}
      \mathbf{0}_{N} \\ u
     \end{array}
 \right)$ with $u\in \mathbb{R}^{M}$. Then for any unit vector $u$,
$$
   \left\langle L(G) \left(
                              \begin{array} {c}
                                 \mathbf{0}_{N} \\ u
                              \end{array}
                            \right),
                           \left(
                              \begin{array} {c}
                                 \mathbf{0}_{N} \\ u
                               \end{array}
                            \right)
   \right\rangle  =  \langle D_2  u, u \rangle \le \delta(G).
$$
This proves our first statement  part that $\lambda_{N-1}(G) \ge N \ge
\delta(G) \ge \lambda_{N+1}(G)$.

When $\lambda_{N}(G) \ge N$, we assert that the degree of any
vertex in the clique $V_1$ is at least $ N $. For this, suppose
that our assertion is false, namely that there exists a vertex
$v_0 \in V_1$ with degree less than $N$. Then this vertex $v_0$ is
adjacent to none of the vertices of the co-clique $V_2$.
Consequently $G$ can be regarded as a split graph with new clique
$V_1 \setminus \{v_0\}$ and new co-clique $V_2 \cup \{v_0 \}$. The
size of the new clique is $\widetilde{N}=N-1$. Applying the first
part of the proposition, we obtain that
$$
\lambda_N(G)=\lambda_{\widetilde{N}+1}(G) \le \widetilde{N}=N-1,
$$
which is a contradiction.

For a conjugating pair of non-negative
integral sequences, the partial sum of one sequence can be
computed in a different way as
$$
\sum_{i=1}^N d_i' = \sum_{i=1}^N \sum_{j=1}^{M+N} \chi(d_j \ge
                      i) = \sum_{j=1}^{M+N} \min(d_j, N),
$$
where $\chi$ is the characteristic function. The second part
of the proposition now follows from the observation that
$$
\sum_{j=1}^{M+N} \min(d_j, N) = \sum_{j\in V_1} N + \sum_{j\in
             V_2} d_j  = N^2 + \mathrm{Tr}(D_2) = N^2 +  \mathrm{Tr}(D_1). \qedhere
$$
\end{proof}

The next lemma will play an essential role in our proof of the
Grone-Merris Conjecture. Its proof is presented in the next section.

\begin{lemma} \label{lem: key lemma}
Assume that $G$ is a split graph of clique size $N$. If either $
\lambda_N (G)> N$ or $ \lambda_N(G)=N> \delta(G)$, then the $N$-th
inequality of the Grone-Merris Conjecture holds, namely
$$
\sum_{i=1}^N \lambda_i \le \sum_{i=1}^N d_i'.
$$
\end{lemma}

\section {The homotopy method}

This section is devoted to proving Lemma \ref{lem: key lemma}. We
adopt a homotopy method, following an idea of Katz \cite{Katz05}
in his proof of the Grone-Merris Conjecture for $1$-regular
semi-bipartite graph.

Let $\alpha \in [0,1]$. Define an $(M+N) \times (M+N) $ matrix
$L_{\alpha}$ as
$$
   L_{\alpha}=(1-\alpha) \left(
                            \begin{array} {cc}
                             K_N + M          & -J_{N \times M}  \\
                             -J_{M \times N}  & N
                         \end{array}
                         \right)
            + \alpha   \left(
                            \begin{array} {cc}
                              K_N + D_1 & -A  \\
                              -A^T      & D_2
                            \end{array}
                       \right),
$$
where $J_{M\times N}$ denotes the $M\times N$ matrix whose entries
are all equal to $1$.

Note that $L_1=L(G)$ is the matrix we are interested in, and that
$L_0$ is the Laplacian of a complete split graph. The spectrum of
$L_0$ is well-understood:

\begin{lemma} \label{lem: complete split}
The Laplacian spectrum of  the complete split graph of clique
size $N$ and co-clique size $M$  is
$$
 \{ \  (M+N)^{(N)},  N^{(M-1)}, 0^{(1)}\ \},
$$
where $P^{(Q)}$ denotes $Q$ copies of the number $P$.
The eigenspace corresponding to the eigenvalue $N$ consists of all vectors
of the form $\left(
     \begin{array} {c}
      \mathbf{0}_{N} \\ v
     \end{array}
 \right)$,
where $v$ is $M$-dimensional and $v\perp \mathbf{1}_M$; the
eigenspace corresponding to the eigenvalue $ (M+N) $ is spanned by
the orthogonal vectors
$$
(\mathbf{0}_{i-1},\  M+N-i, \  -\mathbf{1}_{M+N-i} )^T, \quad 1\le i \le N.
$$
\vskip -\belowdisplayskip
\vskip - \baselineskip
\qed
\end{lemma}

\begin{lemma} \label{lem: ev bounds}
If $\lambda_N (G)> N $ or $ \lambda_N(G)=N > \delta(G) $, then
$$
\lambda_{N+1}^{( \alpha )}  \le N < \lambda_{N}^{( \alpha )}
\text{ for all } 0\le \alpha <1.
$$
\end{lemma}

\begin{proof}
We again make use of  the Courant-Fischer-Weyl Min-Max Principle.
 Recall that the $M$-dimensional subspace $Q \subset
\mathbb{R}^{M+N}$ consists of all vectors of the form
$\left(
     \begin{array} {c}
      \mathbf{0}_{N} \\ u
     \end{array}
 \right)$ with $u\in \mathbb{R}^{M}$. Then for any unit vector $u$,
\begin{align*}
   \left\langle L_{\alpha} \left(
                              \begin{array} {c}
                                 \mathbf{0}_{N} \\ u
                              \end{array}
                            \right),
                           \left(
                              \begin{array} {c}
                                 \mathbf{0}_{N} \\ u
                               \end{array}
                            \right)
   \right\rangle & = (1-\alpha) \langle N u, u \rangle + \alpha \langle D_1(u) , u \rangle \\
                 & \le (1-\alpha)N+ \alpha \delta(G)
                  \le N.
\end{align*}
Therefore, the $(N+1)$-th largest eigenvalue $\lambda_{N+1}^{(
\alpha )}$ is at most $N$.

For the eigenvalue $\lambda_{N}^{( \alpha )}$, let $\tilde{P}$ be
the $N$-dimensional subspace  which is spanned by the eigenvectors
of $L_1$ corresponding to the $N$ largest eigenvalues. Clearly
$\tilde{P} \perp \mathbf{1}_{M+N}$. For any unit vector $v \in
\tilde{P}$,  we know from Lemma \ref{lem: complete split} that
$\langle L_0 (v), v \rangle \ge N$. Moreover,
\begin{align*}
   \langle L_{\alpha} (v), v \rangle
                      & = \alpha \langle L_1 (v), v \rangle
                               + (1-\alpha) \langle L_0 (v), v
                               \rangle \\
                      & \ge \alpha \, \lambda_N(G) + (1-\alpha)N
                       \ge N.
\end{align*}
Therefore,  the $N$-th largest eigenvalue $\lambda_{N}^{( \alpha
)}$ is at least $N$.

We next proceed to show that the inequality on $\lambda_{N}^{(
\alpha )}$ is strict, when $0\le \alpha <1$.
We already know that $\lambda_{N}^{( 0)}=M+N$. If $\lambda_{N}^{(
\alpha )}=N$ for some $0< \alpha <1$, then the above arguments
show that necessarily
$$
\lambda_N(G) =N, \ \ \langle L_1 v, v \rangle=N, \text{\ and \ }
L_0(v)=Nv.
$$
The first condition $\lambda_N(G) =N$ implies that  $\delta(G) <
N$, from our assumption on $\lambda_N(G)$; the third condition
$L_0(v)=Nv$ implies that $v$ is a unit vector in
$\mbox{Ker}(L_0-N)$, thus in turn a unit vector of $ Q $. Then
$$
\langle L_1 v, v \rangle \le \delta(G) < N,
$$
 which contradicts the second condition $\langle L_1 v, v \rangle=N$.
\end{proof}

We now consider all possible $N$-dimensional subspaces
$\left(
     \begin{array} {c}
      I_N \\ V^{(\alpha)}
     \end{array}
  \right) \subseteq (\mathbf{1}_{M+N})^\perp
$, where $V^{(\alpha)}$ is an $M\times N$ matrix.  Here the
notation of the subspace means that the subspace is spanned by the
column vectors of the matrix
$\left(
     \begin{array} {c}
      I_N \\ V^{(\alpha)}
     \end{array}
  \right)
$.

\begin{lemma} \label{lem: V equation}
If the subspace $\left(
     \begin{array} {c}
      I_N \\ V^{(\alpha)}
     \end{array}
  \right) \subseteq (\mathbf{1}_{M+N})^\perp
$ is an invariant subspace of $L_{\alpha}$, then the matrix
$V^{(\alpha)}$ solves the quadratic matrix equation
\begin{align*}
  & V^{(\alpha)} \left[ (1-\alpha)M + \alpha (N+ D_1) \right]  \\
= & -(1-\alpha) J_{M\times N} -\alpha A^T + \alpha  \left[  D_2 -
V^{(\alpha)} (J_{N\times M}-A
    ) \right] V^{(\alpha)}.
\end{align*}
In terms of matrix entries, this means that
\begin{equation}
\label{eqn:eqn1}
v_{ji}^{(\alpha)}=\frac { -(1-\alpha) - \alpha  \chi(i\sim j)  +
\alpha \left(f_j v_{ji}
        - \sum_{i'=1}^N \sum_{j' \nsim i'} v_{ji'}^{(\alpha)} v_{j' i}^{(\alpha)}
        \right) }
        { (1-\alpha)M + \alpha (N+d_i)},
\end{equation}
where the non-negative integers $d_i$, $f_j$ are the entries of the diagonal matrices
$$
D_1=\mathrm{Diag}(d_1,d_2, \ldots,d_N), \quad D_2=\mathrm{Diag}(f_1, f_2,
\ldots,f_M).
$$
\end{lemma}

\begin{proof}
It is easy to see that the orthogonal complement in
$\mathbb{R}^{M+N}$ of the subspace $ \left(
     \begin{array} {c}
      I_N \\ V^{(\alpha)}
     \end{array}
 \right)
$ is the subspace
$
\left(
   \begin{array} {c}
       -{V^{(\alpha)}}^T \\  I_{M}
     \end{array}
  \right)
$.
If the subspace
$
\left(
     \begin{array} {c}
      I_N \\ V^{(\alpha)}
     \end{array}
  \right)
$ is an invariant subspace of $L_{\alpha}$, then so is its
orthogonal complement, since $L_\alpha$ is a symmetric matrix.

The $L_\alpha$-invariance property  is equivalent to the
existence of two square matrices $X_{\alpha}$ and
$Y_{\alpha}$ such that
$$
    L_{\alpha} \left(
                     \begin{array} {cc}
                        I_N          & -{V^{(\alpha)}}^T \\
                        V^{(\alpha)} & I_{M}
                     \end{array}
              \right)
             = \left(
                         \begin{array} {cc}
                                I_N          & -{V^{(\alpha)}}^T \\
                                V^{(\alpha)} & I_{M}
                           \end{array}
               \right)
               \left(
                         \begin{array} {cc}
                             X_{\alpha} & 0 \\
                             0          & Y_{\alpha}
                         \end{array}
              \right).
$$
By comparison of the corresponding four block matrices, we immediately
obtain that
$$
X_{\alpha}=
     K_N+(1-\alpha)M+\alpha D_1 -[ (1-\alpha)J_{N\times M} +\alpha A
     ]V^{(\alpha)},
$$
$$
Y_{\alpha}= (1-\alpha)N +\alpha D_2 + [ (1-\alpha)J_{M\times N}
              +\alpha A^T ]{V^{(\alpha)}}^T,
$$
together with a quadratic matrix equation for $V^{(\alpha)}$:
\begin{align*}
  & V^{(\alpha)} \left[ K_N+(1-\alpha)M + \alpha D_1 \right] +
  (1-\alpha) J_{M\times N} + \alpha A^T \\
= & \left\{(1-\alpha ) N + \alpha D_2 + V^{(\alpha)}
\left[(1-\alpha
    )J_{N\times M} +\alpha A \right] \right\} V^{(\alpha)}.
\end{align*}

Because $ \left(
     \begin{array} {c}
      I_N \\ V^{(\alpha)}
     \end{array}
 \right) \perp \mathbf{1}_{M+N}
$, the entries of $V^{(\alpha)}$ satisfy that
$$
 \sum_{j=1}^M v_{ji}^{(\alpha)}=-1 \text{  for any  }
1\le i\le N.
$$
This condition, in terms of matrices, is equivalent to $
J_{N\times M} V^{(\alpha)}= -J_{N\times N}. $ This implies that $
V^{(\alpha)}K_N = [N +V^{(\alpha)}J_{N\times M}] V^{(\alpha)} $,
with which the above quadratic matrix equation can be simplified
to
\begin{align*}
  & V^{(\alpha)} [(1-\alpha)M + \alpha (N+ D_1)]  \\
= & -(1-\alpha) J_{M\times N} -\alpha A^T + \alpha  \left[  D_2 -
V^{(\alpha)} (J_{N\times M}-A
    ) \right] V^{(\alpha)}. \qedhere
\end{align*}

\end{proof}

The quadratic matrix equation is complicated, and is almost
impossible to be solved explicitly. Fortunately, we do not have to
do so.

From Lemma \ref{lem: ev bounds} and the assumption on
$\lambda_N(G)$, we know that
$$
\lambda_{N+1}^{(\alpha)} \lneqq \lambda_{N}^{(\alpha)} \quad
\mbox{for all} \ \ \alpha \in [0,1].
$$
Thus the subspace spanned by the eigenvectors of $L_{\alpha}$
corresponding to the $N$ largest eigenvalues is unique. Assume that
this subspace is given by $\left(
     \begin{array} {c}
      I_N \\ V^{(\alpha)}
     \end{array}
  \right)
$, so that the matrix $V^{(\alpha)}$ is well defined.

\begin{lemma}\label{lem: continuity of eigenspace}
The map $V^{(\alpha)}: [0,1]\rightarrow \mathbb{R}^{M \times N}$
is a continuous function of $\alpha$, for the usual metric of
$\mathbb{R}^{M\times N}$.
\end{lemma}

\begin{proof}Assume that $\alpha_n$ is a
sequence in $[0,1]$ such that $\alpha_n \rightarrow \alpha$ as
$n\rightarrow \infty$.

According to the algebraic multiplicity of eigenvalues of
$L_{\alpha}$, there exist positive integers $l=l(\alpha)$ and
$i_1,\ldots, i_l$ ($i_0=0$ by convention) such that $
i_1+i_2+\dots +i_l=N $ and
$$
 \lambda_{i_1+\dots + i_{k-1}+1}^{(\alpha)}=\dots
 =\lambda_{i_1+\dots + i_{k-1}+ i_k}^{(\alpha)}
> \lambda_{1+i_1+\dots + i_{k-1}+ i_k}^{(\alpha)}, \quad \forall
1\le k \le l.
$$

Let $\{ u_i^{\beta} \}_{i=1}^{M+N}$ be an orthonormal basis
consisting of the eigenvectors corresponding to the eigenvalues
$\lambda_i^{(\beta)}$ for any $\beta\in [0,1]$, and
$\{Z_{k}^{\alpha_n}\}_{k=1}^l$, $ \{ W_k^{\alpha} \}_{k=1}^l$
denote two sequences of monotonic subspaces of $\mathbb{R}^{M+N}$
given by
$$
Z_{k}^{\alpha_n}  = \mbox{span} \{ u_i^{\alpha_n}: i \le i_1+
\dots + i_{k} \}, \   W_k^{\alpha}  = \mbox{span} \{ u_i^{\alpha}:
i> i_1+ \dots + i_{k-1}  \}.
$$

By the Courant-Fischer-Weyl Min-Max Principle,
$$
\underset{0\ne u \in Z_k^{\alpha_n} } {\mbox{min} } \frac{ \langle
L_{\alpha_n}(u), u \rangle }{\langle u, u \rangle } =
\lambda_{i_1+ \dots + i_{k}}^{(\alpha_n)} \rightarrow
\lambda_{i_1+ \dots + i_{k}}^{(\alpha)} \quad \mbox{as} \ n
\rightarrow \infty
$$
and
$$
\underset{0\ne v \in W_{k+1}^\alpha} {\mbox{max} } \frac{ \langle
L_{\alpha}(v), v \rangle }{\langle v, v \rangle } =
\lambda_{1+i_1+ \dots + i_{k} }^{(\alpha)} \lvertneqq
\lambda_{i_1+ \dots + i_{k} }^{(\alpha)}.
$$
It follows that $Z_k^{\alpha_n} \cap W_{k+1}^\alpha = \{ 0 \}$ and
$ Z_k^{\alpha_n} \oplus W_{k+1}^\alpha = \mathbb{R}^{M+N}$ when
$n$ is sufficiently large. Moreover, we obtain that $
Z_{l}^{\alpha_n} = \bigoplus_{k=1}^l \left(Z_k^{\alpha_n}  \cap
W_k^\alpha \right) $ from
$$
\mbox{dim}(Z_k^{\alpha_n} \cap W_k^\alpha)=
\mbox{dim}(Z_k^{\alpha_n})+\mbox{dim}(W_k^\alpha)- (M+N)= i_k.
$$

Consider a basis of the subspace $Z_k^{\alpha_n} \cap W_k^\alpha$
which consists of unit vectors of the form
$$
u_{k,n,s}=\cos(\theta_{k,n,s})u_{i_1+ \dots + i_{k-1}+s}^{\alpha}
+\sin(\theta_{k,n,s})w_{k,s}, \quad 1\le s \le i_k,
$$
for some unit vector $w_{k,s}\in W_{k+1}^\alpha$. Necessarily
$\lim_{n \rightarrow \infty} \sin(\theta_{k,n,s}) = 0$, since $
\langle L_{\alpha_n}(u_{k,n,s}), u_{k,n,s} \rangle \ge
\lambda_{i_1+ \dots + i_{k} }^{(\alpha_n)} $ and
\begin{align*}
\langle L_{\alpha}(u_{k,n,s}), u_{k,n,s} \rangle = &
\cos^2(\theta_{k,n,s})\lambda_{i_1+ \cdots + i_{k} }^{\alpha} +
\sin^2(\theta_{k,n,s})\langle L_{\alpha}( w_{k,s} ), w_{k,s} \rangle \\
\le & \cos^2(\theta_{k,n,s})\lambda_{i_1+ \cdots + i_{k}
}^{(\alpha)} + \sin^2(\theta_{k,n,s}) \lambda_{i_1+ \cdots + i_{k}
+1 }^{(\alpha)}.
\end{align*}

Any vector $u\in Z_{l}^{\alpha_n}$ can now be expressed as
$$
u=\sum_{k=1}^l \sum_{s=1}^{i_k} c_{k,s} \left[
\cos(\theta_{k,n,s})u_{i_1+ \dots + i_{k-1}+s}^{\alpha}
+\sin(\theta_{k,n,s})w_{k,s} \right].
$$

Assume that the maximum of $|c_{k,s}|$ is achieved at $|c_{k_0,
s_0}|$.  Due to the orthogonality of $\{ u_i^{\alpha}\}_i$, the
absolute value of the coefficient of $u_{i_1+ \dots +
i_{k_0-1}+s_0}^{\alpha}$ is at most $\Vert u \Vert$. But when $n$
is sufficiently large, it is at least
$$
|c_{k_0, s_0}| \cdot \left( |\cos(\theta_{k_0,n,s_0})|-
\sum_{k=1}^l \sum_{s=1}^{i_k} |\sin(\theta_{k,n,s})| \right) \ge
\frac{|c_{k_0, s_0}|}{2}.
$$
Hence $ |c_{k_0, s_0}| \le 2 \Vert u \Vert $.
 For any given vector $v \in W_{l+1}^\alpha$, we see that
$$
| \langle u,v \rangle |  = \left| \sum_{k=1}^l
\sum_{s=1}^{i_k}  \langle c_{k,s} \sin(\theta_{k,n,s}) w_{k,s},  v
\rangle \right|
  \le 2 \Vert u \Vert \cdot \Vert v
\Vert \cdot \sum_{k=1}^l
\sum_{s=1}^{i_k} |\sin(\theta_{k,n,s})| ,
$$
which goes to zero as $n$ goes to infinity.

The subspace $Z_{l}^{\alpha_n}$ is nothing else but $\left(
     \begin{array} {c}
      I_N \\ V^{(\alpha_n)}
     \end{array}
  \right)
$, while $W_{l+1}^\alpha$ is nothing else but $\left(
     \begin{array} {c}
     -{V^{(\alpha)}}^T  \\ I_M
     \end{array}
  \right)
$. The inner product of the $i$-th column vector of the first
matrix and the $j$-th column vector of the second matrix is equal
to
$$
V_{ji}^{(\alpha_n)}- V_{ji}^{(\alpha)},
$$
which must go to zero as $n$ goes to infinity. This proves the
continuity of $V^{(\alpha)}$ on $\alpha$.
\end{proof}

\begin{lemma} \label{lem: connected interval}
Let $\Omega$ be the subset
$$
\{(x_{ji}): \sum_{k=1}^M x_{ki}=-1, \  \forall \,  1\le i\le N,
\mbox{\ and \ } x_{ji} \le 0, \,\,\, \forall \, 1\le j \le M, 1\le
i\le N \}.
$$
of $\mathbb{R}^{M\times N}$.
Then $V^{(\alpha)} \in \Omega$ for all $\alpha \in [0,1]$.
\end{lemma}

\begin{proof} Consider the subset
$$
\Gamma=\{ \alpha\in [0,1):  V^{(\alpha)} \in \Omega\}
$$
 of the half-open half-closed interval $[0,1)$.

When $\alpha=0$,  $ v_{ji}^{(0)} \equiv -\frac{1}{M} $ (see Lemma
\ref{lem: complete split} or Equation (\ref{eqn:eqn1}) ). As a
consequence, $V^{(0)} \in \Omega$, so that  $0\in \Gamma$ and
$\Gamma$ is not empty.

Suppose there is a sequence of points $\alpha_n \in \Gamma$ and
$\lim_{n \rightarrow \infty} \alpha_n = \alpha$ with $\alpha$
still in $[0,1)$. By Lemma \ref{lem: continuity of eigenspace},
$\lim_{n \rightarrow \infty} V^{(\alpha_n)}=V^{(\alpha)}$. Because
$\Omega$ is a compact set, so $V^{(\alpha)} \in \Omega$ and
$\alpha \in \Gamma$. Therefore, $\Gamma$ is a closed   subset of
$[0,1)$.

Suppose $\alpha \in \Gamma$, namely $V^{(\alpha)} \in \Omega$ for
some $\alpha \in [0,1)$. Because  the quantities $\chi(i\sim j)$,
$f_j$ and $v_{ji'}^{(\alpha)} v_{j' i}^{(\alpha)}$ in Equation
(\ref{eqn:eqn1}) are all non-negative, we see that
$$
v_{ji}^{(\alpha)} \le \frac{-(1-\alpha)}{(1-\alpha)M
                         + \alpha (N+d_i)}  <0
 \text{ for all }  1\le j \le M, 1\le i\le N.
$$
Therefore $V^{(\alpha)}$ is contained in the interior of $\Omega$.
Since $V^{(\alpha)}$ depends continuously on $\alpha$, it follows
that $\Gamma$ is an open subset of $[0,1)$.

The interval $[0,1)$ is connected, and $\Gamma$ is an open closed
non-empty subset of it, therefore $\Gamma$ is equal to $[0,1)$.

By continuity at $\alpha=1$, $V^{(1)}$ is also in $\Omega$. This
proves that $V^{(\alpha)} \in \Omega$ for all $\alpha \in [0,1]$.
\end{proof}

During the proof of Lemma \ref{lem: V equation}, we have already
known that
$$
   L_{\alpha}   \left(
                      \begin{array} {c}
                       I_N \\ V^{(\alpha)}
                      \end{array}
               \right)
   =          \left(
                      \begin{array} {c}
                          I_N  \\ V^{(\alpha)}
               \end{array}
               \right)
        X_{\alpha}
$$
where
$$
X_{\alpha}= K_N+(1-\alpha)M+\alpha D_1 -[ (1-\alpha)J_{N\times M}
            +\alpha A ]V^{(\alpha)}.
$$
So the sum of the $N$ largest eigenvalues of $L_1$ is equal to the
trace of
$$
 X_1= K_N+ D_1 -  A V^{(1)}.
$$
But $V^{(1)} \in \Omega$ by Lemma \ref{lem: connected interval},
therefore
$$
\mathrm{Tr}(AV^{(1)})=\sum_{i=1}^N \sum_{j: j\sim i} v_{ji}
\ge \sum_{i=1}^N \sum_{j=1}^M v_{ji} =-N .
$$
Then
$$
\sum_{i=1}^N \lambda_i = N(N-1) + \mathrm{Tr}(D_1) -
\mathrm{Tr}(AV^{(1)}) \le N^2 + \mathrm{Tr}(D_1).
$$
By Proposition \ref{prop: simple fact}, this completes the proof of
Lemma \ref{lem: key lemma}.

\section {Proof of Grone-Merris Conjecture}
\label{sec: main proof}

For consistence we restate the Grone-Merris Conjecture here.
\begin{GMconjecture}
For any graph $G$, its Laplacian spectrum is majorized by its
conjugate degree sequence, namely
$
 \lambda(G) \preccurlyeq {\mathbf d'}(G)
$.
\end{GMconjecture}

The Grone-Merris Conjecture behaves nicely under complementation,
in the sense of the proposition below.

The \textit{complement graph} of a graph $G$ is a
graph $\overline{G}$ on the same vertices such that two vertices
of $\overline{G}$ are adjacent if and only if they are not
adjacent in $G$. The Laplacian matrices of the graph $G$ and of its
complementary graph $\overline{G}$ are related by the property that
$$
 L(G)+L(\overline{G})+J_n = nI_n.
$$
All these matrices commute with  each other,
so that
$$
\lambda(\overline{G}) =(n-\lambda_{n-1}(G), \ldots,
n-\lambda_1(G), 0);
$$
$$
 \mathbf{d}'(\overline{G}) =(n-d_{n-1}'(G), \ldots, n-d_1'(G), 0).
$$
From these we see that

\begin{proposition}\label{prop: complement}
For any $1\le k < n$, the $k$-th inequality holds for the graph
$G$ if and only if the $(n-k-1)$-th inequality holds for the
complement graph $\overline{G}$.
$$
\sum_{i=1}^k \lambda_i(G) \le \sum_{i=1}^k d_i'(G)
\Longleftrightarrow \sum_{j=1}^{n-1-k} \lambda_j(\overline{G}) \le
\sum_{j=1}^{n-1-k} d_j'(\overline{G}), \quad \forall 1\le k < n.
$$
\vskip -\belowdisplayskip
\vskip -\baselineskip
\qed
\end{proposition}

We are now ready to prove the Grone-Merris Conjecture.

Assume that the Grone-Merris Conjecture is not true, and the graph
$G=(V,E)$ is a counterexample. Namely, there exists an integer $k$
with  $1< k <n=|V|$, such that
$$
\sum_{i=1}^k \lambda_i > \sum_{i=1}^k d_i'.
$$

Without loss of generality, we can assume that this integer $k$ is
minimum over all counterexamples. Then we have
$$
\sum_{i=1}^{k-1} \lambda_i \le \sum_{i=1}^{k-1} d_{i}', \quad
\mbox {and} \quad \lambda_k > d_k'.
$$

Moreover, we can further assume that the number $|E|$ of edges is
minimum over all counterexamples with the same $k$. Under this
assumption, we claim that

\begin{lemma}
For any two vertices $i,j$ in the graph $G$,  if $d_i \le k $ and
$d_j \le k$, then they are not adjacent in $G$.
\end{lemma}

\begin{proof}
We will prove this by contradiction. Assuming that the lemma is
false, namely there exists a pair of vertices such that
$$
d_i \le k, \quad d_j\le k, \quad i\sim j .
$$
Let $\widetilde{G}$ be
the graph obtained from $G$ by deleting the edge $ij$. Due to the
minimum property of $|E|$, we must have
$$
\sum_{i=1}^k \lambda_i(\widetilde{G}) \le \sum_{i=1}^k
d_{i}'(\widetilde{G}).
$$

Two Laplacian matrices are related via $ L(G)=L(\widetilde{G}) + H
$, where $H_{n \times n}$ is a positive semi-definite matrix whose
only non-zero entries are $H_{ii}=H_{jj}=1$ and
$H_{ij}=H_{ji}=-1$. Applying Fan's Theorem \ref{thm: Fan Ky}, we
see that
\begin{align*}
\sum_{i=1}^{k} \lambda_i(G)
       & \le  \sum_{i=1}^{k}
                       \lambda_i(\widetilde{G}) + \sum_{i=1}^{k} \lambda_i(H)
        \le \sum_{i=1}^{k} d_i'(\widetilde{G}) + Tr(H) \\
       & = \left[ \sum_{i=1}^{k} d_i'(G) -2 \right] +2
       = \sum_{i=1}^{k} d_i'(G).
\end{align*}

This contradicts our assumption that $G$ was a
counterexample, and therefore concludes the proof.
\end{proof}

Next, we add new edges to $G$ to get a new graph $\widehat{G}$.
Add to $G$ a new edge $ij$ for any pair of vertices $i$, $j$ in $G$ such that
$$
d_i \ge k, \ d_j \ge k,   \text{ and }  i\nsim j .
$$
The new graph $\widehat{G}$ so obtained  is a split graph.

The clique of $\widehat{G}$ consists of all vertices of $G$ whose
degree is at least $k$, so the size of the clique is equal to
$d_k'(G)$. Let  $N=d_{k}'(G)$ denote this size. The co-clique
consists of all vertices of $G$ whose degree is less than $k$, so
the maximum  degree of vertices in the co-clique is
$\delta(\widehat{G})\le k-1$.

Note that
$$
  d_1'(\widehat{G})=d_1'(G), \dots, d_k'(\widehat{G})=d_k'(G)
$$
  while
$ \lambda_i(\widehat{G}) \ge \lambda_i(G)$
      for all  $ 1\le i \le n
$, so these two inequalities are still valid for the new graph
$\widehat{G}$, namely
$$
\sum_{i=1}^k \lambda_i(\widehat{G}) > \sum_{i=1}^k
d_i'(\widehat{G}) \quad \mbox{and} \quad \lambda_k(\widehat{G})
> d_k'(\widehat{G})=N.
$$

Let us discuss the relationship between $N$ and $k$.

If $N<k$, then $\lambda_k(\widehat{G}) \le
\lambda_{N+1}(\widehat{G}) \le N$, which leads to a contradiction.
The second inequality comes  from Proposition \ref{prop: simple
fact}.

If $N=k$, then $\widehat{G}$ is a split graph of clique size $N$,
with the  property that
$$
\sum_{i=1}^N \lambda_i(\widehat{G})  > \sum_{i=1}^N
d_i'(\widehat{G}) \quad \mbox{and} \quad \lambda_N(\widehat{G})
> N.
$$
This  contradicts Lemma \ref{lem: key lemma}.

So $k< N$. Note that $\widehat{G}$ is a split graph of clique
size $N$. In this graph $\widehat{G}$, the maximum degree of
vertices in the  co-clique is  at most $(k-1)$, while the minimum
degree of  vertices in the clique is at least $(N-1)$. This means
that
$$
d_{N-1}'(\widehat{G})=\cdots=d_{k+1}'(\widehat{G})=d_k'(\widehat{G})=N.
$$

Combining this with $\lambda_{k+1}(\widehat{G}) \ge \ldots \ge
\lambda_{N-1}(\widehat{G}) \ge N$ from Proposition \ref{prop:
simple fact}, we see immediately that the inequality
$$
\sum_{i=1}^{k} \lambda_i(\widehat{G})
> \sum_{i=1}^{k} d_i'(\widehat{G})  \text{ can be extended to }
\sum_{i=1}^{N-1} \lambda_i(\widehat{G})
> \sum_{i=1}^{N-1} d_i'(\widehat{G}).
$$

Then we proceed to compare  $\lambda_N(\widehat{G})$ with the clique size $N$.

First consider the case where  $\lambda_N(\widehat{G}) \ge N$.
Because $N=d _{N-1}'(\widehat{G}) \ge d_N'(\widehat{G})$, the
split graph $\widehat{G}$ has clique size $N$, with the additional
property that
$$
\sum_{i=1}^N \lambda_i(\widehat{G})  > \sum_{i=1}^N
d_i'(\widehat{G}) \quad \mbox{and} \quad \lambda_N(\widehat{G})
\ge N > \delta(\widehat{G}).
$$
This again contradicts Lemma \ref{lem: key lemma}.

In the other case, where $\lambda_N(\widehat{G}) < N$, we switch
attention to the complement graph of $\widehat{G}$. This
complement graph is another split graph $\overline{\widehat{G}}$.
Its clique size is $M$, and
$$
\lambda_M(\overline{\widehat{G}})=(N+M) -\lambda_N(\widehat{G}) > M.
$$
According to Proposition \ref{prop: complement},
$$
\sum_{i=1}^{N-1} \lambda_i(\widehat{G})
> \sum_{i=1}^{N-1} d_i'(\widehat{G}) \quad \Longrightarrow \quad
 \sum_{i=1}^M \lambda_i (\overline{\widehat{G}}) > \sum_{i=1}^M
d_i' (\overline{\widehat{G}}).
$$
Therefore, $\overline{\widehat{G}}$ is a split graph of clique
size $M$, with the additional property that
$$
\sum_{i=1}^M \lambda_i (\overline{\widehat{G}}) > \sum_{i=1}^M
d_i' (\overline{\widehat{G}}) \quad \mbox{and} \quad
\lambda_M(\overline{\widehat{G}})> M.
$$
This again  contradicts Lemma \ref{lem: key lemma}.

All possible cases are eliminated, and the Grone-Merris Conjecture
is proved.

\end{document}